\documentclass[12pt,a4paper]{amsart}

\usepackage{amssymb,amsmath,amsthm}
\usepackage{jipam}

\newcommand{\RR}{\mathbb{R}}

\title[Higher Order Convexity]{Inequalities and Higher Order Convexity}
\author{Zarathustra Brady}
\email{notzeb@caltech.edu}

\keywords{convexity, inequalities}
\subjclass[2000]{39B62}

\begin{document}

\begin{abstract}
We study the following problem: given $n$ real arguments $a_1, ..., a_n$ and $n$ real weights $w_1, ..., w_n$, under what conditions does the inequality
\[
w_1f(a_1) + w_2f(a_2) + \cdots + w_nf(a_n) \ge 0
\]
hold for all functions $f$ satisfying $f^{(k)}\ge0$ for some given integer $k$? Using simple combinatorial techniques, we can prove many generalizations of theorems ranging from the Fuchs inequality to the criterion for Schur convexity.
\end{abstract}

\maketitle

\section{Introduction}

The theory of majorization is remarkably rich and complete, culminating in Karamata's inequality \cite{karamata}, Muirhead's inequality \cite{muirhead}, and the theory of Schur-convex functions \cite{schur} (Karamata's inequality is also known as the Hardy-Littlewood-Polya inequality \cite{inequalities}). It therefore seems natural to try to generalize it to functions with a higher order of convexity. In particular, we study the following problem:

Given $n$ real arguments $a_1, ..., a_n$ and $n$ real weights $w_1, ..., w_n$, under what conditions does the inequality
\[
w_1f(a_1) + w_2f(a_2) + \cdots + w_nf(a_n) \ge 0
\]
hold for all functions $f$ satisfying $f^{(k)}\ge0$ for some given integer $k$?

\section{Basics of higher order convex functions}

Higher order convexity was introduced by Popoviciu, who defined it in terms of the divided differences of a function. Divided differences are defined inductively as follows:
\begin{align*}
[a_1; f] =& f(a_1)\\
[a_1, ..., a_{n+1}; f] =& \frac{[a_1, ..., a_n; f] - [a_2, ..., a_{n+1}; f]}{a_1 - a_{n+1}}.
\end{align*}
He then defined a $k-1$th order convex function to be one with all $k+1$ order divided differences positive. If the function has a $k$th derivative, then it is known that such functions are exactly those for which $f^{(k)} > 0$ \cite{convex}. For simplicity, we will only deal with functions with nonnegative $k$th derivative, although these results can be extended to all $k-1$th order convex functions.

Intuitively, the prototypical increasing function is a step function, while the prototypical convex function is the absolute value function. More generally, the prototypical functions with positive $k+1$th derivative are functions such as the following:

\begin{definition}
\[
(x)_+^k := \begin{cases}
x^k&\text{if } x>0,\\
0&\text{if } x\leq 0.\end{cases}
\]
\end{definition}

\begin{lemma}\label{classification}
For all $f:[a,b] \to \RR$ with $f^{(k)} \geq 0$, there is a nondecreasing function $\alpha:[a,b] \to \RR$ and a $k-1$th degree polynomial $P$ such that
\[
f(x) = P(x) + \int_{t=a}^b (x-t)_+^{k-1} d\alpha(t).
\]
\end{lemma}
\begin{proof}
We prove by induction that we can in fact take $\alpha(t) = \frac{f^{(k-1)}}{(k-1)!}$. For the base case $k = 1$, we have
\[
\int_{t=a}^b (x-t)_+^0 \, df(t) = \int_{t=a}^x df(t) = f(x) - f(a).
\]
Now assume that it is proven for $k$, and we will prove it for $k+1$:
\begin{align*}
&\frac1{k!}\int_{t=a}^b (x-t)_+^k \, df^{(k)}(t)\\
= &\frac1{k!}\left[(x-t)_+^kf^{(k)}(t)\right]_{t=a}^b - \frac1{k!}\int_{t=a}^b f^{(k)}(t) \, d(x-t)_+^k\\
= &-\frac1{k!}(x-a)^kf^{(k)}(a) + \frac{k}{k!}\int_{t=a}^b (x-t)_+^{k-1} \, df^{(k-1)}(t),
\end{align*}
where the last integral is equal to $f(x)$ plus a polynomial of degree $k-1$ by the induction hypothesis.
\end{proof}

\begin{remark}
This lemma can be extended to arbitrary $k-1$th order convex functions if we interpret the integral as a Stieltjes integral, but the proof is more technical (see Bullen's paper, Corollary 8 \cite{criterion}).
\end{remark}

Our strategy from here on is to reduce proving an inequality on arbitrary $k-1$th order convex functions to proving it for the prototypical functions of the form $(x-t)_+^{k-1}$.

\begin{definition}
Given real arguments $a_1, a_2, ..., a_n$, real weights $w_1, w_2, ..., w_n$, and a positive integer $k$, let
\[
r_k(x) := \sum_{i=1}^n w_i (a_i - x)_+^{k-1}.
\]
\end{definition}

\begin{lemma}\label{auxiliaries}
Given real arguments $a_1, a_2, ..., a_n \in [a,b]$, real weights $w_1, w_2, ..., w_n$, and a positive integer $k$,
\[
\sum_{i=1}^n w_i f(a_i) \geq 0
\]
for all functions $f:[a,b] \to \RR$ with $f^{(k)} \geq 0$ iff
\[
\sum_{i=1}^n w_i a_i^j = 0 \text{ for integers } 0 \leq j < k \text{, and}
\]
\[
r_k(x) \geq 0 \text{ for all } x.
\]
\end{lemma}
\begin{proof}
Setting $f(x) = \pm x^j$ for $1 \leq j < k$ and $f(x) = (x-t)_+^{k-1}$, we see that both conditions are necessary. For the other direction, let $P$ and $\alpha:[a,b] \to \RR$ be as in Lemma \ref{classification}. Then we have
\begin{align*}
\sum_{i=1}^n w_i f(a_i) &= \sum_{i=1}^n w_i \left(P(a_i) + \int_{t=a_n}^{a_1} \! (a_i-t)_+^{k-1} d\alpha(t)\right)\\
&= \sum_{i=1}^n w_i P(a_i) + \int_{t=a_n}^{a_1} \left(\sum_{i=1}^n w_i (a_i-t)_+^{k-1}\right) d\alpha(t)\\
&= \int_{t=a_n}^{a_1} \! r_k(t) \, d\alpha(t) \geq 0. \qedhere
\end{align*}
\end{proof}

Using this lemma, it is \emph{in principle} possible to test the truth of an inequality by carefully analyzing $r_k(x)$ (using, say, the theory of Sturm chains). For instance, a complete classification of inequalities on functions with $f''' \ge 0$ is given by the following:

\begin{theorem}\label{useless}
Given real arguments $a_1 > a_2 > \cdots > a_n \in [a,b]$, and real weights $w_1, w_2, ..., w_n$ such that
\[
\sum_{i=1}^n w_i = \sum_{i=1}^n w_ia_i = \sum_{i=1}^n w_ia_i^2 = 0,
\]
a necessary and sufficient condition for the inequality
\[
\sum_{i=1}^n w_if(a_i) \ge 0
\]
to be true for every function $f:[a,b] \to \RR$ with $f''' \ge 0$ is that
\[
(\sum_{i=1}^j w_i)(\sum_{i=1}^j w_ia_i^2) \ge (\sum_{i=1}^j w_ia_i)^2
\]
for all $j$ such that $(\sum_{i=1}^j w_i)a_j \ge \sum_{i=1}^j w_ia_i \ge (\sum_{i=1}^j w_i)a_{j+1}$.
\end{theorem}
\begin{proof}
By Lemma \ref{auxiliaries}, we just need to find the condition for $r_3$ to be nonnegative on $[a,b]$. The conditions imply that $r_3$ is $0$ outside of the interval $[a,b]$, so if it is ever negative then it must have a local minimum in some interval $[a_{j+1},a_j]$. $r_3$ is quadratic on each interval, so we must then have $r_3'(a_{j+1}) \le 0 \le r_3'(a_j)$, and the minimum of $r_3$ on this interval can be easily computed (the details are left to the reader).
\end{proof}

\section{A simple trick}

A little experimentation with small cases indicates that, generally, $r_k(x)$ is always positive or always negative when $n$ is small - there just aren't enough variables for it to change sign. Formally,

\begin{definition}
The number of \emph{sign changes} of a function $f:[a,b] \to \RR$, is the maximum number $n$ such that there exist real numbers $a_1 \geq a_2 \geq \cdots \geq a_{n+1} \in [a,b]$ with $f(a_i)f(a_{i+1}) < 0$ for $1 \leq i \leq n$.
\end{definition}

\begin{lemma}\label{wiggles}
If a differentiable nonconstant function $f:[a,b] \to \RR$ is such that $f(a) = f(b) = 0$, then either $f$ has strictly fewer sign changes than $f'$, or both $f$ and $f'$ have an infinite number of sign changes.
\end{lemma}
\begin{proof}
This is an easy consequence of the mean value theorem.
\end{proof}

Now for the main result:

\begin{theorem}\label{counting}
Given real arguments $a_1 > a_2 \geq a_3 \geq \cdots \geq a_n \in [a,b]$, real weights $w_1, w_2, ..., w_n$, and a positive integer $k$ such that $w_1 > 0$,
\[
\sum_{i=1}^n w_i a_i^j = 0 \text{ for all integers } 0 \leq j < k,
\]
and such that one of the following three conditions is satisfied:
\begin{itemize}
\item There are at most $k$ sign changes in the sequence $w_1, w_2, ..., w_n$.
\item There are at most $k-1$ sign changes in the sequence $w_1, w_1+w_2, ..., w_1+\cdots+w_n$.
\item There are at most $k-2$ sign changes in the sequence $w_1a_1 - w_1a_1, (w_1a_1 + w_2a_2) - (w_1 + w_2)a_2, ..., (w_1a_1 + \cdots + w_na_n) - (w_1 + \cdots + w_n)a_n$.
\end{itemize}
Then we have
\[
\sum_{i=1}^n w_i f(a_i) \geq 0
\]
for all functions $f:[a,b] \to \RR$ with $f^{(k)} \geq 0$.
\end{theorem}
\begin{proof}
From the assumption $\sum_{i=1}^n w_i a_i^j = 0$, we see that $r_j(a) = r_j(b) = 0$ for all $1 \leq j \leq k$.

It's easy to see that the first condition implies the second condition, by a discrete analog of Lemma \ref{wiggles}.

To understand the second condition, note that $r_1$ is a step function taking the values $w_1, w_1+w_2, ..., w_1+\cdots+w_n$, so the number of sign changes in this sequence is the same as the number of sign changes of $r_1$.

To understand the third condition, note that $r_2$ is a piecewise linear function taking the value $0$ outside the interval $[a,b]$ and that the values of $r_2$ at the points where its slope changes are given by the sequence $w_1a_1 - w_1a_1, (w_1a_1 + w_2a_2) - (w_1 + w_2)a_2, ..., (w_1a_1 + \cdots + w_na_n) - (w_1 + \cdots + w_n)a_n$, so the number of sign changes of this sequence is the same as the number of sign changes of $r_2$.

Now, since $r_j'(x) = jr_{j-1}(x)$, we get from repeated application of Lemma \ref{wiggles} that for each $j$, the number of sign changes of $r_j(x)$ is at most $k-j$, so $r_k$ has $0$ sign changes. Thus, since $r_k(a_2) = w_1(a_1-a_2)^{k-1} > 0$, $r_k(x)$ must be nonnegative for all $x$. At this point, we simply apply Lemma \ref{auxiliaries} to finish the proof.
\end{proof}

\begin{remark}
Based on the proof, we can also see that if
\[
\sum_{i=1}^n w_i a_i^j = 0 \text{ for all integers } 0 \leq j < k,
\]
then there can't be fewer than $k-1$ sign changes in the sequence $w_1, w_1+w_2, ..., w_1+\cdots+w_n$.
\end{remark}

\section{Applications}

\begin{corollary}\label{smooth1}
Given real numbers $a_1, a_2, ..., a_k \in [a,b]$ and $b_1, b_2, ..., b_k \in [a,b]$ such that
\[
\sum_{i=1}^k a_i^j = \sum_{i=1}^k b_i^j \text{ for all integers } 1 \leq j < k,
\]
the following are equivalent:
\begin{enumerate}
\item $\sum_{i=1}^k a_i^k \geq \sum_{i=1}^k b_i^k\\$
\item $\max \{a_i\}_{i=1}^k \geq \max \{b_i\}_{i=1}^k\\$
\item $\sum_{i=1}^k f(a_i) \geq \sum_{i=1}^k f(b_i) \text{ for all functions } f:[a,b] \to \RR$ \text{ with } $f^{(k)} \geq 0.$
\end{enumerate}
\end{corollary}

The $k = 3$ case of Corollary \ref{smooth1} was originally proved by Vasile C{\^{\i}}rtoaje, using the identity
\[
f(x) + f(y) + f(z) - f(a) - f(b) - f(c)=\frac 1{2}(xyz-abc)f'''(w)
\]
for some $w$ in the smallest interval containing all of $a, b, c, x, y, z$ \cite{vasc}.

In general, it's easy to see that the third bullet point of Corollary \ref{smooth1} implies the other two bullet points by plugging in $f(x) = x^k$ or $f(x) = (x-\max\{a_i\}_{i=1}^k)_+^{k-1}$, and that since we can't have equality in either of the first two bullet points when the sets $\{a_i\}_{i=1}^k, \{b_i\}_{i=1}^k$ are different (by the fundamental theorem of algebra applied to the polynomials $\prod_{i=1}^k(x-a_i)$ and $\prod_{i=1}^k(x-b_i)$) the reverse implications will hold if we can prove that $r_k$ has no sign changes. Thus, Corollary \ref{smooth1} follows from the argument of the next Corollary upon setting $a_{k+1} = b_{k+1}$.

\begin{corollary}\label{smooth2}
Given real numbers $a_1 \geq a_2 \geq \cdots \geq a_{k+1} \in [a,b]$ and $b_1 \geq b_2 \geq \cdots \geq b_{k+1} \in [a,b]$ with $a_1 \ge b_1$ and $(-1)^k a_{k+1} \le (-1)^k b_{k+1}$ such that
\[
\sum_{i=1}^{k+1} a_i^j = \sum_{i=1}^{k+1} b_i^j \text{ for all integers } 1 \leq j < k,
\]
we have
\[
\sum_{i=1}^{k+1} f(a_i) \geq \sum_{i=1}^{k+1} f(b_i)
\]
for all functions $f:[a,b] \to \RR$ with $f^{(k)} \geq 0$.
\end{corollary}
\begin{proof}
Let $c_1 \geq c_2 \geq \cdots \geq c_{2k+2}$ be the union of the $a_i$s and the $b_i$s, and let $w_j$ be either $1$ or $-1$, depending on whether $c_j$ was originally an $a_i$ or a $b_i$. By Theorem \ref{counting}, it's enough to show that the sequence $w_1, w_1+w_2, ..., w_1+\cdots+w_{2k+2}$ has at most $k-1$ sign changes.

Assume, for the sake of contradiction, that $w_1, w_1+w_2, ..., w_1+\cdots+w_{2k+2}$ has at least $k$ sign changes. Then we must have a sequence $i_1 \leq i_2 \leq \cdots \leq i_{k+1}$ such that $(w_1+\cdots+w_{i_j})(w_1+\cdots+w_{i_{j+1}}) < 0$ for all $j$. Then $i_{j+1} \geq i_j+2$, so $2k+2 > i_{k+1} \geq i_k+2 \geq \cdots \geq i_1+2k \geq 2k+1$ ($i_{k+1} \neq 2k+2$, because $w_1+\cdots+w_{2k+2} = 0$). Thus, we must have $i_1 = 1$, $i_{k+1} = 2k+1$, and $(-1)^kw_1(w_1+\cdots+w_{2k+1}) > 0$. On the other hand, since $c_1 = a_1$, we have $w_1 = 1 > 0$, and similarly $w_1+\cdots+w_{2k+1} = -w_{2k+2} = (-1)^{k+1}$, a contradiction.
\end{proof}

A similar argument gives us the following Corollary, but instead we will directly derive it from Corollary \ref{smooth2}:

\begin{corollary}\label{gensen2}
Given real arguments $a_1 > a_2 > \cdots > a_{k+2} \in [a,b]$ and weights $w_1, w_2, ..., w_{k+2}$  with $w_1 \geq 0$ and $(-1)^kw_{k+2} \geq 0$ such that
\[
\sum_{i=1}^{k+2} w_i a_i^j = 0 \text{ for all } 0 \leq j < k,
\]
we have
\[
\sum_{i=1}^{k+2} w_i f(a_i) \geq 0
\]
for all functions $f:[a,b] \to \RR$ with $f^{(k)} \geq 0$.
\end{corollary}
\begin{proof}
Let $p_1, p_2 ..., p_{k+2}:[-\epsilon, \epsilon] \to \RR$ be monotone, differentiable functions such that $p_i(0) = a_i$, $p_1'(0) = w_1$, $p_{k+2}'(0) = w_{k+2}$, and
\[
\sum_{i=1}^{k+2} p_i(x)^j
\]
is constant for integers $1 \leq j < k+1$. Then, since $p_1$ is increasing and $(-1)^{k+1}p_{k+2}$ is decreasing, we can apply Corollary \ref{smooth2} to see that the function $\alpha:[-\epsilon, \epsilon] \rightarrow \RR$ given by
\[
\alpha(x) = \sum_{i=1}^{k+2} F(p_i(x))
\]
is increasing for any functions $F:[a,b] \to \RR$ with $F^{(k+1)} \geq 0$. Differentiating with respect to $x$ at $0$, we get that
\[
\sum_{i=1}^{k+2} p_i'(0)F'(a_i) \geq 0.
\]
Since the $p_i'(0)$s satisfy the same $k$ independent linear equations as the $w_i$s, $p_1'(0) = w_1$, and $p_{k+2}'(0) = w_{k+2}$, we get that $p_i'(0) = w_i$ for all $i$. Now we just take $F$ such that $F' = f$ to finish the proof.
\end{proof}

Just for fun, let's derive the weighted analogue of Corollary \ref{smooth1} with a completely elementary method (i.e., without depending on Lemma \ref{classification}):

\begin{corollary}\label{gensen}
Given real arguments $a_1 > a_2 > \cdots > a_{k+1} \in [a,b]$ and weights $w_1, w_2, ..., w_{k+1}$  with $w_1 \geq 0$ such that
\[
\sum_{i=1}^{k+1} w_i a_i^j = 0 \text{ for all } 0 \leq j < k,
\]
we have
\[
\sum_{i=1}^{k+1} w_i f(a_i) \geq 0
\]
for all functions $f:[a,b] \to \RR$ with $f^{(k)} \geq 0$.
\end{corollary}
\begin{proof}
We use induction on $k$. When $k = 1$, the Corollary follows trivially from the fact that $f$ is increasing, while for $k = 2$ it is just a restatement of the weighted Jensen inequality on two variables. Assume that it is true for $k-1$. We can assume that $a_{k+1} = 0$ without loss of generality. Now, define a new function $h:[a,b] \to \RR$ by
\[
h(t) := \sum_{i=1}^{k+1} w_i f(ta_i).
\]
Since $h(0) = 0$, it is enough to show that $h'(t) \geq 0$ for $0 \leq t \leq 1$. But $h'(t)$ is just
\[
\sum_{i=1}^k w_ia_i f'(ta_i),
\]
since $a_{k+1} = 0$. Now we see that the arguments $ta_1 \geq ta_2 \geq \cdots \geq ta_k$, the weights $w_1a_1, w_2a_2, ..., w_ka_k$ and the function $f'$ satisfy the conditions of the Corollary for $k-1$, so we can apply the induction hypothesis to see that $h'(t) \geq 0$.
\end{proof}

Corollary \ref{gensen} is, in fact, a restatement of the fact that $[a_1, ..., a_{k+1}; f]$ is positive for functions with positive $k$th derivative. Using only this fact, one can prove the following theorem, which can be found in Popoviciu's book \cite{convex}:

\begin{theorem}\label{Big Hammer}
Given real arguments $a_1 > a_2 > \cdots > a_n \in [a,b]$, real weights $w_1, w_2, ..., w_n$, and a given integer $k$ such that
\[
\sum_{i=1}^n w_i a_i^j = 0 \text{ for all } 0 \leq j < k \text{, and}
\]
\[
\sum_{i=1}^j w_i (a_i - a_{j+1})(a_i - a_{j+2})\cdots(a_i - a_{j+k-1}) \geq 0
\text{ for all } 1 \leq j \leq n-k,
\]
we have
\[
\sum_{i=1}^n w_i f(a_i) \geq 0
\]
for all functions $f:[a,b] \to \RR$ with $f^{(k)} \geq 0$.
\end{theorem}
\begin{proof}
The theorem follows directly from the following identity:
\[
\sum_{j=1}^n w_j f(a_j) = \sum_{j=1}^{n-k}(a_j-a_{j+k})
\left(\sum_{i=1}^j w_i (a_i - a_{j+1})\cdots(a_i - a_{j+k-1})\right)[a_j, ..., a_{j+k}; f].
\]
The proof of this identity is left as an exercise to the reader.
\end{proof}

\begin{remark}\label{counter}
When $k = 2$, Theorem \ref{Big Hammer} is equivalent to what Darij Grinberg calls the weighted Karamata inequality \cite{darij}. Plugging in the functions $f(x) = (x-a_{j+1})_+$, we see that in the case $k = 2$, the condition given is both necessary and sufficient.

Since the condition from Corollary \ref{gensen2} is both necessary and sufficient, and equivalent to the condition in Theorem \ref{Big Hammer} whenever $n \leq k+2$, we can easily see that any inequality not following from Theorem \ref{Big Hammer} must have $k \geq 3$, $n \geq k+3 \geq 6$. A simple example of such an inequality with $k = 3$, $n = 6$ is the following:
\[
f(6) - 3f(5) + 3f(4) - 3f(2) + 3f(1) - f(0) \geq 0,
\]
which, though it doesn't satisfy the conditions of Theorem \ref{Big Hammer}, can easily be seen to be true by adding $-f(3) + f(3)$ to the left hand side, giving the equivalent inequality $6[6, 5, 4, 3; f] + 6[3, 2, 1, 0; f] \geq 0$.
\end{remark}

The condition given in Theorem \ref{Big Hammer}, although easier to check than the condition of Lemma \ref{auxiliaries}, is still inconvenient because of the need to order all of the variables involved. On the other hand, to apply inequalities such as the Karamata inequality or the Fuchs inequality \cite{fuchs}, one only needs to order the variables occurring on each side of the inequality before applying the corresponding theorem. The next theorem is an attempt to find a simpler condition for inequalities on functions with nonnegative third derivative.

\begin{theorem}\label{small hammer}
Given weights $w_1, ..., w_n$, real numbers $a_1 \geq \cdots \geq a_n \in [a,b]$, and real numbers $b_1 \geq \cdots \geq b_n \in [a,b]$ such that $\min(a_i, b_i) \geq \max(a_{i+1}, b_{i+1})$,
\[
\sum_{i=1}^n w_ia_i = \sum_{i=1}^n w_ib_i \text{, }
\sum_{i=1}^n w_ia_i^2 = \sum_{i=1}^n w_ib_i^2 \text{, and}
\]
\[
\sum_{i=1}^j w_i(a_i-a_{j+1})(a_i-b_{j+1}) \geq \sum_{i=1}^j w_i(b_i-a_{j+1})(b_i-b_{j+1})
\]
for all $1 \leq j < n$, we have
\[
\sum_{i=1}^n w_if(a_i) \geq \sum_{i=1}^n w_if(b_i).
\]
for all functions $f:[a,b] \to \RR$ with $f''' \ge 0$.
\end{theorem}
\begin{proof}[Proof 1, using Theorem \ref{Big Hammer}]
Notice that the last set of conditions make up every other condition from the second set of conditions of Theorem \ref{Big Hammer}. We can assume without loss of generality that $b_j = \min(a_j, b_j)$, and $a_{j+1} = \max(a_{j+1}, b_{j+1})$ (by swapping them and negating the weights, if necessary). We would like to prove that, given
\begin{align}
\label{eq1}\sum_{i=1}^j w_i(a_i-a_j)(a_i-b_j) \geq& \sum_{i=1}^j w_i(b_i-a_j)(b_i-b_j) \text{ and}\\
\label{eq2}\sum_{i=1}^j w_i(a_i-a_{j+1})(a_i-b_{j+1}) \geq& \sum_{i=1}^j w_i(b_i-a_{j+1})(b_i-b_{j+1}) \text{, we have}\\
\label{eq3}\sum_{i=1}^j w_i(a_i-b_j)(a_i-a_{j+1}) \geq& \sum_{i=1}^j w_i(b_i-b_j)(b_i-a_{j+1}).
\end{align}
Suppose first that we have the inequality
\begin{equation}
\tag{*}\label{*} \sum_{i=1}^j w_ia_i \geq \sum_{i=1}^j w_ib_i,
\end{equation}
and note that it is equivalent to the inequality
\[
(a_j-a_{j+1})\sum_{i=1}^j w_i(a_i-b_j) \geq (a_j-a_{j+1})\sum_{i=1}^j w_i(b_i-b_j).
\]
Adding this to (\ref{eq1}), we get (\ref{eq3}).

Similarly, suppose that we have the opposite inequality
\begin{equation}
\tag{**}\label{**} \sum_{i=1}^j w_ia_i \leq \sum_{i=1}^j w_ib_i,
\end{equation}
and note that it is equivalent to the inequality
\[
(b_{j+1}-b_j)\sum_{i=1}^j w_i(a_i-b_j) \geq (b_{j+1}-b_j)\sum_{i=1}^j w_i(b_i-b_j).
\]
Adding this to (\ref{eq2}), we get (\ref{eq3}).

Since at least one of the two inequalities (\ref{*}), (\ref{**}) is true, (\ref{eq3}) must be true if both (\ref{eq1}) and (\ref{eq2}) are true.
\end{proof}
\begin{proof}[Proof 2]
Applying Abel Summation twice, we find that
\begin{align*}
&\sum_{i=1}^n w_i(f(a_i)-f(b_i))\\
= &\sum_{j=1}^n \left(\sum_{i=1}^j w_ia_i -w_ib_i\right)
\left(\frac{f(a_j)-f(b_j)}{a_j-b_j}-\frac{f(a_{j+1})-f(b_{j+1})}{a_{j+1}-b_{j+1}}\right)\\
= &\sum_{j=1}^n \left(\sum_{i=1}^j w_i(a_i-a_{j+1})(a_i-b_{j+1})-w_i(b_i-a_{j+1})(b_i-b_{j+1})\right) \times\\
&\left(\frac{\frac{f(a_j)-f(b_j)}{a_j-b_j}-\frac{f(a_{j+1})-f(b_{j+1})}{a_{j+1}-b_{j+1}}}{a_j+b_j-a_{j+1}-b_{j+1}}
-\frac{\frac{f(a_{j+1})-f(b_{j+1})}{a_{j+1}-b_{j+1}}-\frac{f(a_{j+2})-f(b_{j+2})}{a_{j+2}-b_{j+2}}}{a_{j+1}+b_{j+1}-a_{j+2}-b_{j+2}}\right).
\end{align*}
Thus, it suffices to show that
\[
\frac{\frac{f(a_j)-f(b_j)}{a_j-b_j}-\frac{f(a_{j+1})-f(b_{j+1})}{a_{j+1}-b_{j+1}}}{a_j+b_j-a_{j+1}-b_{j+1}} \geq
\frac{\frac{f(a_{j+1})-f(b_{j+1})}{a_{j+1}-b_{j+1}}-\frac{f(a_{j+2})-f(b_{j+2})}{a_{j+2}-b_{j+2}}}{a_{j+1}+b_{j+1}-a_{j+2}-b_{j+2}}.
\]
Plugging in $f(x) = 1$, $x$, and $x^2$, we get equality. Assume, as in the first proof, that $a_j \ge b_j \ge \cdots \ge b_{j+2}$, collect everything on the left hand side of the inequality, and let $w_{2i}$ be the weight on $f(a_{j+i})$, and $w_{2i+1}$ be the weight on $f(b_{j+i})$. Since $w_{2i} + w_{2i+1} = 0$, only three of the partial sums of the weights are nonzero, so there are at most 2 sign changes among the partial sums of the weights, and thus by Theorem \ref{counting} we are done.
\end{proof}

The corresponding unweighted inequality (proved by applying the previous theorem with integer weights) is:

\begin{corollary}\label{superize}
Given real numbers $a_1 \geq \cdots \geq a_n$, and real numbers $b_1 \geq \cdots \geq b_n$ such that either $\min(a_i, b_i) \geq \max(a_{i+1}, b_{i+1})$ or $\{a_i, b_i\} = \{a_{i+1}, b_{i+1}\}$ for all $i$,
\[
\sum_{i=1}^n a_i = \sum_{i=1}^n b_i \text{, }
\sum_{i=1}^n a_i^2 = \sum_{i=1}^n b_i^2 \text{, and}
\]
\[
\sum_{i=1}^j (a_i-a_j)(a_i-b_j) \geq \sum_{i=1}^j (b_i-a_j)(b_i-b_j) \text{ for all } j,
\]
we have
\[
\sum_{i=1}^n f(a_i) \ge \sum_{i=1}^n f(b_i).
\]
\end{corollary}

\begin{example}
The inequality $f(11) + f(8) + f(8) + f(7) + f(3) + f(1) \geq f(10) + f(10) + f(6) + f(6) + f(6) + f(0)$, for functions with $f''' \ge 0$, doesn't follow directly from Corollary \ref{superize} (because $(a_3, b_3) = (8, 6)$ and $(a_4, b_4) = (7, 6)$), but if we add $f(7)$ to both sides, then we can apply Corollary \ref{superize} to check that it is true.
\end{example}

\section{Maximal and minimal expressions}

Fix a number $k \ge 3$, a number $n \ge k$, and reals $s_1, ..., s_{k-1}$, and let $S_k$ be the set $\{(x_1, ..., x_n) \in \mathbb{R}^n | x_1 \ge x_2 \ge \cdots \ge x_n, \sum_{i=1}^n x_i^j = s_j \text{ for } j = 1, ..., k-1\}$, with the topology induced from $\mathbb{R}^n$. We can define a partial ordering on $S$ by

\begin{definition}
Let $a = (a_1, ..., a_n)$, $b = (b_1, ..., b_n)$, then we say $a \succ_k b$ if, for all functions $f$ with $f^{(k)} \ge 0$, we have
\[
\sum_{i=1}^n f(a_i) \ge \sum_{i=1}^n f(b_i).
\]
\end{definition}

An immediate question that comes to mind about this ordering is this one: are there elements $x_{min}, x_{max} \in S_k$ such that for all $x \in S_k$, $x_{max} \succ_k x \succ_k x_{min}$? If so, what do they look like? Since any $x$ determines the set $S_k$ containing it, we speak of maximal and minimal $x$ without specifying $S_k$ explicitly (other than mentioning the value of $k$).

First, let's see what happens when $n = k$:

\begin{lemma}\label{1d}
If $x \in S_k$ locally maximizes the function $\sum_{i=1}^k x_i^k$, then there exists an $i$ such that $x_{2i} = x_{2i+1}$, and if $x$ locally minimizes it then there exists an $i$ such that $x_{2i-1} = x_{2i}$.
\end{lemma}
\begin{proof}
Consider the polynomial $p(t) = (t-x_1)(t-x_2)\cdots(t-x_k)$. If we vary the constant term of this polynomial while keeping the other coefficients fixed, then as long as it still has $k$ real roots $y_1, ..., y_n$ we will have $y \in S_k$ and $y \succ_k x$ iff the constant term was increased (by Corollary \ref{smooth1} and Newton's identities). Then we can increase the constant term while keeping all of the roots real iff there are no double roots at which the second derivative of the polynomial is $\ge 0$, and can decrease the constant term iff there are no double roots at which the second derivative is $\le 0$, and these conditions are clearly equivalent to the conditions given in the statement of the Lemma.
\end{proof}

Note that if an element $x \in S_k$ is maximal with respect to our ordering (no longer restricting ourselves to the case $n = k$), then for any $k$ integers $1 \le i_1 < \cdots < i_k \le n$, $x' = (x_{i_1}, ..., x_{i_k})$ must be maximal in the set $S_k'$ containing it, and thus by Lemma \ref{1d} and Corollary \ref{smooth1}, there is some integer $j$ such that $x_{i_{2j}} = x_{i_{2j+1}}$ (the corresponding statement is also true for minimal $x$, with the roles of even and odd indices reversed). Now, we can classify all maximal elements as follows:

\begin{theorem}\label{extreme}
For any $x \in S_k$, the following conditions are equivalent:
\begin{enumerate}
\item $x$ locally maximizes the function $\sum_{i=1}^n x_i^k$ in $S_k$.
\item There exist integers $1 = i_1 \le \cdots \le i_k = n+1$ such that $x_{i_j} = x_{i_j+1} = \cdots = x_{i_{j+1}-1}$ and $i_{2j} - i_{2j-1} \in \{0, 1\}$ for each $j$.
\item For all $y \in S_k$, $x \succ_k y$.
\end{enumerate}
\end{theorem}
\begin{proof}
$(1) \implies (2)$ For any integer $i$, let $s(i)$ be the smallest integer larger than $i$ such that $x_{s(i)} \ne x_i$ (or $n+1$ if no such number exists). We inductively form the sequence $i_1, ..., i_k$ by setting $i_1 = 1$, and $i_{j+1} = s(i_j)$ unless $j$ is odd and $s(i_j) \le n$, in which case we set $i_{j+1} = i_j+1$. Then, if $i_k \ne n+1$, we immediately see that the vector $(x_{i_1}, ..., x_{i_k})$ is not maximal in the set $S_k'$ containing it, a contradiction. The sequence $i_1, ..., i_k$ satisfies the required conditions.

$(2) \implies (3)$ Surprisingly, this is nothing more than an application of Theorem \ref{counting}! Let $sc(i)$ be the number of sign changes in the partial sums of the weights up to (but not including) the first partial sum containing the weight corresponding to $x_i$. Then it's easy to check that $sc(i_2) \le 1$, $sc(i_{2j+2}) \le sc(i_{2j}) + 2$, and, if k is odd, $sc(i_k) \le sc(i_{k-1})+1$ (they all follow from the fact that the partial sums are integers, so it takes at least two consecutive weights of the same sign for the partial sum to cross $0$). Thus, $sc(i_k) \le k-1$, so all that we have left to check for Theorem \ref{counting} is that $x_1 > y_1$ if $x \ne y$. But this must be the case, because otherwise $sc(i_2) = 0$, so $sc(i_k) < k-1$, contradicting the remark following Theorem \ref{counting}.

$(3) \implies (1)$ Obvious.
\end{proof}

A similar classification applies to minimal elements, with the roles of even and odd reversed once again. In the $k = 3$ case, this theorem is a special case of the $n-1$ Equal Variable Principle, due to Vasile C{\^{\i}}rtoaje, which states that a maximal element $x$ has $x_1 \ge x_2 = \cdots = x_n$, while a minimal element has $x_1 = \cdots = x_{n-1} \ge x_n$, even in a more general setting where $S$ is defined by fixing the sum of first powers and $p$th powers (of course, the condition that the third derivative is positive is replaced with a different condition: that $f'(x^{\frac{1}{p-1}})$ is convex) \cite{equal}.

We also have the following unexpected bonuses:

\begin{corollary}\label{connect}
$S_k$ is a connected set.
\end{corollary}
\begin{proof}
Every connected subset of $S_k$ is compact, and thus contains a point $x$ which is a local maximum of $\sum_{i=1}^n x_i^k$. Thus, this point $x$ is maximal with respect to our ordering on $S_k$. All that's left is to prove that maximal elements are unique (because then any two connected sets contain the same maximal element $x$). Suppose $y$ is another maximal element, let $i$ be the first integer such that $x_i \ne y_i$, and consider the function $f(z) = (z-\frac{x_i+y_i}{2})_+^{k-1}$ to find a contradiction.
\end{proof}

\begin{corollary}\label{point}
If $n \ge k$, then $S_k$ consists a single point iff $S_k$ contains a point which is maximal or minimal in $S_{k-1} = \{(x_1, ..., x_n) \in \mathbb{R}^n | x_1 \ge x_2 \ge \cdots \ge x_n, \sum_{i=1}^n x_i^j = s_j \text{ for } j = 1, ..., k-2\}$.
\end{corollary}
\begin{proof}
One direction is obvious. The other direction follows from a restatement of the second condition in Theorem \ref{extreme}: to any vector $x$, we can associate a string of $a$s and $b$s by assigning each block of equal coordinates of $x$ an $a$ if it has length $1$, and a $ba$ if it has length more than $1$. Then $x$ is extremal iff there is a way to insert $a$s and $b$s such that the transformed string is alternating $a$s and $ba$s and contains at most $k-1$ $a$s (maximal if it starts with an $a$, minimal if it starts with a $b$). If $S_k$ contains only one point, then it must be both maximal and minimal, so its corresponding string is a substring of both $abaa...a$ and $baaba...a$ containing at least one $b$ (if there was no $b$, then it would contain $n$ $a$s). But any maximal common substring of those containing at least one $b$ is one of the two such strings with $k-2$ $a$s (we can prove this by induction: either it's an $a$ followed by a maximal common substring of the two such strings with $k-2$ $a$s, or it doesn't start with an $a$, in which case it's a substring of $abaaba...$ without the first $a$. In either case, it's a substring of one of those two strings with $k-2$ $a$s.) Thus it's either maximal or minimal for $k-1$.
\end{proof}

\section{Increasing paths}

Another natural question to ask about $S_k$, digressing from our main aim of solving inequalities, is the question of whether $x \succ_k y$ implies that $x$ and $y$ are connected by an increasing path (the last section can easily be used to prove this when $x$ is maximal or $y$ is minimal). For $k = 1$ this is obvious, and for $k = 2$, it follows from the theory of majorization. It also holds when $n = k$, as is easily seen from the fact that $S_k$ is connected along with Corollary \ref{smooth1}. We might guess that this is in fact true for all $n, k$, offering as evidence proofs in the cases $n = k+1$ and $k = 3$:

\begin{theorem}\label{path2}
If $n = k+1$, then for any $a, b \in S_k$ such that $a \succ_k b$, there is a continuous function $p:[0, 1] \to S_k$ such that $p(0) = b, p(1) = a$ and $p(t_1) \succ_k p(t_0)$ for $t_1 \ge t_0$.
\end{theorem}
\begin{proof}
Let $q:[0, 1] \to S_k$ be any increasing path from $b$ to a maximal element of $S_k$, and let $t$ be the first time such that either the largest coordinate of $q(t)$ equals the largest coordinate of $a$ or the smallest coordinate of $q(t)$ equals the smallest coordinate of $a$. Then by Corollary \ref{smooth2}, we still have $a \succ_k q(t)$, so we can find an increasing path connecting $q(t)$ and $a$ (by canceling the equal coordinates and applying the fact that there is always an increasing path when $n = k$).
\end{proof}

\begin{theorem}\label{path}
For any $a, b \in S_3$ such that $a \succ_3 b$, there is a continuous function $p:[0, 1] \to S_3$ such that $p(0) = b, p(1) = a$ and $p(t_1) \succ_3 p(t_0)$ for $t_1 \ge t_0$.
\end{theorem}
\begin{proof}
Suppose $a, b$ are a counterexample with minimal $n$. Then $a_i \ne b_j$ for any $i, j$. Let $m$ be the first integer such that $b_1 > a_m$. Define $r_3(x) = \sum_{i=1}^n (x-a_i)_+^2 - \sum_{i=1}^n (x-b_i)_+^2$, and let $x_0$ be the first real number less than $a_1$ such that $r_3(x_0) = 0$. We will prove first that if $r_3(x) \ge 0$ for all $x$, then $a_{m+1} > x_0$:

Let $y$ be the largest real less than $a_1$ such that $r_3'(y) \le 0$, then clearly $b_1 > y > x_0$ and $r_3''(y) < 0$. Also, we must have $r_3'(x_0) = 0$ and $r_3''(x_0) > 0$ (from the nonnegativity of $r_3$), so there must be at least two components of $a$ between $y$ and $x_0$ (since $r_3''$ is always an even integer, and increases by two for every component of $a$). Thus, since $a_{m-1} > b_1 > y$, we have $a_{m+1} = a_{m-1+2} > x_0$.

Thus, we can continuously decrease $(a_{m-1}, a_m, a_{m+1})$ with respect to our ordering without invalidating the inequality (since $r_3$ will stay the same outside the interval $[a_{m-1}, a_{m+1}]$ and never hits zero inside that interval), until one of $a_{m-1}, a_m, a_{m+1}$ is equal to some $b_i$. This happens by the time $(a_{m-1}, a_m, a_{m+1})$ becomes minimal, since $a_{m-1} > b_1 > a_m$. Then we can use induction to find a path from here.
\end{proof}

As a Corollary, we get a slight generalization of Schur Convexity:

\begin{corollary}\label{schur}
For any symmetric function $f: \mathbb{R}^n \to \mathbb{R}$, we have $f(a) \ge f(b)$ whenever $a \succ_3 b$ if and only if
\[
\left(\frac{\frac{\partial f}{\partial x_1}-\frac{\partial f}{\partial x_2}}{x_1-x_2} - \frac{\frac{\partial f}{\partial x_2}-\frac{\partial f}{\partial x_3}}{x_2-x_3}\right)(x_1-x_3) \ge 0
\]
for all $x_1, x_2, x_3 \in \mathbb{R}$.
\end{corollary}
\begin{proof}
The given condition is (locally) equivalent to the condition that
\[
f(x_1, x_2, x_3, a_4, ..., a_n) \succ_3 f(y_1, y_2, y_3, a_4, ..., a_n)
\]
whenever $(x_1, x_2, x_3) \succ_3 (y_1, y_2, y_3)$, and by the proof of Theorem \ref{path} we can form a (finite) chain of inequalities of this form, starting from $a$ and ending at $b$, as long as $a \succ_3 b$.
\end{proof}

Unfortunately, the proofs of the existence of increasing paths for $k = 3$ and $n=k+1$ don't generalize - in both cases, the paths found are built up by changing $k$ variables at a time. For instance, the following class of inequalities can't be proven by following such paths:

\begin{theorem}\label{six}
For any $x, y, z, a, b, c \ge 0$ such that $x^2+y^2+z^2 = a^2+b^2+c^2$ and $x^3+y^3+z^3 = a^3+b^3+c^3$, we have
\[
(x, y, z, -z, -y, -x) \succ_4 (a, b, c, -c, -b, -a)
\]
iff $\max(x, y, z) \ge \max(a, b, c)$.
\end{theorem}
\begin{proof}
The proof is based on Theorem \ref{counting}. There are at most five sign changes in the partial sums of the weights, so $r_3 = \frac{r_4'}{3}$ for this inequality has at most three sign changes. But $r_4$ is symmetric around $0$, so $r_4'$ has at most one sign change in $(0, \infty)$, and the given conditions are equivalent to $r_4(0) = r_4'(0) = 0$, so $r_4$ has no sign changes in $(0, \infty)$, and thus $r_4$ is always the same sign. Since $r_4(\max(a,b,c)) > 0$ when $\max(x, y, z) \ge \max(a, b, c)$, we see that in this case $r_4(t) \ge 0$ for all $t$. (The theorem still holds when $x^3+y^3+z^3 \ge a^3+b^3+c^3$, but we only need this version of the theorem for our counterexample.)
\end{proof}

\begin{example}
Now consider the path $p(t) = (x(t), y(t), z(t), -z(t), -y(t), -x(t))$ defined by the differential equations $\frac{dx}{dt} = \frac{1}{x(x-y)(x-z)}$ (and similarly for $\frac{dy}{dt}, \frac{dz}{dt}$). It's easy to see that then $\frac{d}{dt}(x^2+y^2+z^2) = \frac{d}{dt}(x^3+y^3+z^3) = 0$, and for any function $f$ with $f^{(4)} \ge 0$, we have
\[
\frac{d}{dt} \sum_{i=1}^6 f(p_i(t)) = [x(t), y(t), z(t), 0; f'] + [0, -x(t), -y(t), -z(t); f'].
\]
If there was some increasing path from $p(t)$ to $p(t+\epsilon)$ made by changing only $4$ variables at a time, this would imply that the above expression could be written as a positive linear combination of expressions of the form $[p_i(t), p_j(t), p_k(t), p_l(t); f']$. Plugging in $x(t) = 3, y(t) = 2, z(t) = 1$, we see that then Theorem \ref{Big Hammer} would be sufficient to prove that $f'(3) - 3f'(2) + 3f'(1) - 3f'(-1) + 3f'(-2) - f'(-3) \ge 0$, but we've already seen (Remark \ref{counter}) that this is not the case, contradiction.
\end{example}

\begin{conjecture}
For any integers $n, k$, reals $s_1, ..., s_{k-1}$, and any $a, b \in S_k = \{(x_1, ..., x_n) \in \mathbb{R}^n | x_1 \ge x_2 \ge \cdots \ge x_n, \sum_{i=1}^n x_i^j = s_j \text{ for } j = 1, ..., k-1\}$ such that $a \succ_k b$, there is a continuous function $p:[0, 1] \to S_k$ such that $p(0) = b, p(1) = a$ and $p(t_1) \succ_k p(t_0)$ for $t_1 \ge t_0$.
\end{conjecture}

\bibliography{article}
\bibliographystyle{plain}

\end{document}